\documentclass[11pt]{amsart}

\usepackage[colorlinks=true]{hyperref}
\usepackage{amssymb}
\usepackage[all]{xy}
\usepackage{mathtools}
\usepackage{stackrel}
\usepackage{verbatim}
	
	\usepackage{cancel}

\usepackage{amsmath}

\usepackage{xcolor}

\usepackage{tikz}
\usepackage{tikz-cd}
\usetikzlibrary{calc,matrix,trees}
\tikzset{w/.style={circle, draw,inner sep=1pt},b/.style={circle,draw,fill,inner sep=2pt}, s/.style={rectangle, draw,inner sep=3pt}}

\newcommand{\nocontentsline}[3]{}
\newcommand{\tocless}[2]{\bgroup\let\addcontentsline=\nocontentsline#1{#2}\egroup}

\newcommand{\lra}{\longrightarrow}

\newcommand{\antishriek}{\text{\raisebox{\depth}{\textexclamdown}}}

\newcommand{\bK}{\mathbb{K}}

\newcommand{\bS}{\mathbb{S}}

\newcommand{\cC}{\mathcal{C}}

\newcommand{\cP}{\mathcal{P}}

\newcommand{\Ass}{\textit{Ass}}

\newcommand{\Comm}{\mathrm{Comm}}
\newcommand{\Coder}{\mathrm{Coder}}

\newcommand{\End}{\mathrm{End}}

\newcommand{\Hom}{\mathrm{Hom}}
\newcommand{\id}{\mathrm{id}}

\newcommand{\Pois}{\textit{Pois}}

\newcommand{\Lie}{\textit{Lie}}

\newcommand{\dgMod}{\texttt{dgMod}}

\newtheorem{thm}{Theorem}[section]
\newtheorem{lem}[thm]{Lemma}
\newtheorem{cor}[thm]{Corollary}
\newtheorem{prop}[thm]{Proposition}

\newtheorem*{thm*}{Theorem}
\newtheorem*{cor*}{Corollary}
\theoremstyle{definition}
\newtheorem*{defi*}{Definition}
\newtheorem{defi}[thm]{Definition}
\theoremstyle{remark}

\newtheorem{remark}[thm]{Remark}

\begin{document}
\title{Formality criteria for algebras over operads}
\address{Dipartimento di Matematica, Universit\`a di Pisa, Largo Bruno Pontecorvo 5, 56127 Pisa, Italy}
\email{valerio.melani@unipi.it}
\email{valerio.melani@outlook.com}
\author{Valerio Melani}
\address{Department of Mathematics, KU Leuven, Celestijnenlaan 200B, 3001 Leuven, Belgium}
\email{marcel.rubio@kuleuven.be}
\author{Marcel Rubi\'o}
\begin{abstract}
We study some formality criteria for differential graded algebras over differential graded operads. This unifies and generalizes other known approaches like the ones contained in \cite{Ka} and \cite{Ma}. In particular, we construct general operadic Kaledin classes and show that they provide obstructions to formality. Moreover, we show that an algebra $A$ is formal if and only if its operadic cohomology spectral sequence degenerates at $E_2$.
\end{abstract}
\maketitle

\tableofcontents

\addtocontents{toc}{\protect\setcounter{tocdepth}{2}}

\section*{Introduction}

One says that a differential graded algebra $A$ is \emph{formal} if it is quasi-isomorphic to its cohomology graded algebra $H(A)$. This notion has been studied in multiple occasions. The most famous instance of formality in the mathematical literature is probably the paper \cite{DGMS} by Deligne, Griffiths, Morgan and Sullivan, in which the authors showed that if $X$ is a compact K\"ahler manifold, then the de Rham algebra $\mathrm{DR}(X)$ of $X$ is formal. This has important consequences in the study of the topology of $X$.

The case of formality for differential graded Lie algebras is also a well studied subject. It is well known that dg Lie algebras are fundamental objects in the realm of deformation theory: a guiding principle by Deligne states that every deformation problem in characteristic zero is controlled by a dg Lie algebra. In particular, every dg Lie algebra $L$ can be associated to a natural deformation functor. Moreover, if $L$ is assumed to be formal, the deformation functor is simplified considerably.

In \cite{GM}, Goldman and Millson used the same approach of \cite{DGMS} to prove that the dg Lie algebra of exterior differential forms taking values in certain flat vector bundles is formal. A geometric consequence of this result is that the moduli space of certain representations of the fundamental group of a compact K\"ahler manifold has at most quadratic singularities.

Another very famous manifestation of formality is found in the seminal paper \cite{Ko} of Kontsevich. Consider the algebra $A$ of smooth functions on a differentiable manifold, then Kontsevich proves that the dg Lie algebra of the Hochschild cochain complex of $A$ (with Hochschild differential and Gerstenhaber bracket) is formal. A corollary of this result shows that every finite-dimensional Poisson manifold can be canonically quantized in the sense of deformation quantization.

More recently, Kaledin wrote the short and dense paper \cite{Ka}, where he addressed the question of finding cohomological obstructions to formalities for \emph{families} of associative dg algebras. A concise exposition of Kaledin's work can be found in the paper \cite{Lu}, in which Lunts provides proper foundations and complete proofs for the main results of \cite{Ka}. Kaledin's formality criteria were later used in \cite{KL} and in \cite{Zh} to attack some particular cases of a conjecture of Kaledin and Lehn about the singularities of the moduli space of semistable sheaves on K3 surfaces. 

Given an associative dg algebra $A$, Kaledin's arguments are based on the construction of a certain Hochschild cohomology class of $A$ (called \emph{Kaledin class} in \cite{Lu}). Kaledin is then able to prove that this class measures the obstruction to formality of $A$. In fact, one of the main results in \cite{Ka} states that the algebra $A$ is formal if and only if its Kaledin class vanishes.
It is worth noting that similar results were obtained for commutative dg algebras by Sullivan in \cite{Su} and by Halperin-Stasheff in \cite{HS}.

Inspired by \cite{Ka} and \cite{Lu}, Manetti introduced a different flavour of formality criteria for dg Lie algebras in \cite{Ma}. Manetti's approach differs from Kaledin's in that the main tool in \cite{Ma} is the Chevalley-Eilenberg spectral sequence of a given dg Lie algebra $L$, which by definition is simply the natural spectral sequence computing the classical Chevalley-Eilenberg cohomology of $L$. It was already known that homotopy abelianity of a dg Lie algebra $L$ is equivalent to the degeneration at $E_1$ of the Chevalley-Eilenberg spectral sequence of $L$ (see \cite{Ba}). The main result of \cite{Ma} extends this by proving that formality of $L$ is equivalent to the degeneration at $E_2$ of the Chevalley-Eilenberg spectral sequence.

In view of these formality criteria for different cases, the purpose of the present paper is to generalize formality criteria of both \cite{Ka} and \cite{Ma} to the general case of algebras over a sufficiently nice Koszul operad $\cP$. In particular, we want $\cP$ to be weight graded and thus we will always assume $\cP$ to be reduced. For example, $\cP$ can be any of the usual suspects \textit{Ass, Lie, Com, PreLie, Pois} etc.
In the same fashion as \cite{Lu} and \cite{Ma}, we found it more convenient to use the language of homotopy $\cP$-algebras, or $\cP_\infty$-algebras. For this purpose, we use the Koszul resolutions described in \cite{GK} and \cite{LV}.

For every $\cP$-algebra $A$, we give an elementary construction of a natural class $K_A$ in the $\cP$-operadic cohomology of $A$, which we call \emph{operadic Kaledin class}, in analogy with the choice of terminology of \cite{Lu}. We remark that our construction avoids some of the technicalities of \cite{Ka} and \cite{Lu}, as we don't need to pass through deformations to the normal cone. As expected, the operadic Kaledin class controls formality of $A$. Our first main result can therefore be stated as follows.
\begin{thm*}[see Theorem \ref{thm:formality1}]
	The $\cP$-algebra $A$ is formal if and only if its operadic Kaledin class $K_A$ vanishes.
\end{thm*}

Notice that in the case where $\cP$ is the operad $\Ass$ of associative algebras, our result gives back the Kaledin formality criterion.

We then pass to the study of the natural spectral sequence computing $\cP$-operadic cohomology of the algebra $A$. Generalizing ideas from \cite{Ma}, we construct a distinguished element $e_A$ in the second page $E_2$ of the operadic spectral sequence of the $\cP$-algebra $A$, which we call \emph{operadic Euler class} following the notations in \cite{Ma}. We are then able to link the vanishing of the operadic Kaledin class $K_A$ to the operadic Euler class $e_A$. More precisely, we show the following statement.
\begin{thm*}[see Theorem \ref{thm:kaledin-euler}]
	The Kaledin class $K_A$ vanishes if and only if $d_r(e_A)=0$ for every $r \geq 2$, where $d_r$ denotes the differential in the $r$-th page of the operadic spectral sequence on $A$.
\end{thm*}

As an immediate corollary of both theorems, we find the following formality criteria for the $\cP$-algebra $A$.

\begin{cor*}
	Let $A$ be a $\cP$-algebra. The following are equivalent:
	\begin{enumerate}
		\item $A$ is formal;
		\item the operadic Kaledin class $K_A$ is zero;
		\item $d_r(e_A)=0$ for every $r \geq 2$;
		\item the operadic spectral sequence of $A$ degenerates at $E_2$.
	\end{enumerate}
	
\end{cor*}

Moreover, if we take $\cP=\Lie$ to be the operad of Lie algebras, then the equivalence of items $(1)-(3)-(4)$ is the main result of \cite{Ma}. We remark however that the formality criterion of item $(2)$ appears to be new even in the case of dg Lie algebras.

\

The paper is structured as follows. The first section is devoted to briefly recalling the theory of Koszul (co)operads and homotopy $\cP$-algebras, following the exposition in \cite{LV}.

In Section 2, we define the fundamental notions of minimal and formal $\cP_\infty$-algebras. Moreover, we construct our main tool for the study of formality, which is the operadic Kaledin class. We then obtain our first formality criterion, which is Theorem \ref{thm:formality1}. As mentioned, this generalizes results of \cite{Ka} and \cite{Lu}.

The goal of Section 3 is to extend some constructions and ideas from \cite{Ma} to our more general setting. We start by defining a natural operadic cohomology spectral sequence, which specializes to the Chevalley-Eilenberg spectral sequence whenever $\cP = \Lie$. Then, we define the operadic Euler class: a canonical element of the second page of the operadic spectral sequence. With these new tools, we are able to formulate a different formality criterion, generalizing a theorem of \cite{Ma}. 

In Section 4 we relate the operadic Kaledin class with the operadic Euler class. Then, we summarize our results in Corollary \ref{cor:formality2} in the form of a formality criterion.

In \cite{Ma}, Manetti also includes a very interesting result on formality transfers. In the same way, Section 5 is devoted to an adaptation of the formality transfer theorem in our setting.

In the final section, we treat in more detail the cases where $\cP$ is the operad $\Ass$ or $\Lie$. We explicitly show how one can get back known results of Kaledin and Manetti from Corollary \ref{cor:formality2}. 

\subsection*{Future directions}

We list here some of the topics we chose not to treat in the present version of the paper. We plan to come back to these questions in future works.

As already explained, Manetti studies formality criteria of dg Lie algebra using degenerations of the Chevalley-Eilenberg spectral sequence in \cite{Ma}. This is much in the spirit of \cite{Ba}, where Bandiera shows that the Chevalley-Eilenberg spectral sequence also detects homotopy abelianity. We certainly feel like the arguments of \cite{Ba} can be smoothly extended to prove a homotopy triviality criterion for algebras over a nice enough operad $\cP$.

While studying the uniqueness of $A_\infty$-structures, Kadeishvili provided an intrinsic formality criterion for $A_\infty$-algebras \cite{Kad}. One says that an $A_\infty$-algebra $A$ is \emph{intrinsically formal} if any other $A_\infty$-algebra $A'$ such that $H(A) \cong H(A')$ as $A_\infty$-algebras is quasi-isomorphic to $A$. A generalization to $\cP$-algebras should follow similarly to the results presented in this paper. We thank Andreas Hochenegger for pointing this result to us.

Another direction arises from the work of Kaledin \cite{Ka}, which deals with formality for families of associative dg algebras. Algebraically, the use of families corresponds to working over a non-trivial base, as it is explained in \cite{Lu}. In our opinion, it would be interesting to extend Kaledin's criterion for formality in families to algebras over general operads. We hope that the methods developed in this paper can be useful for this purpose.

Finally, formality has important consequences in deformation theory. From a more geometric point of view, the formality of a dg Lie algebra implies that the associated formal moduli problem has at most quadratic singularities. It would be interesting to study how different formalities impact the associated geometrical object. For example, pre-Lie formality could possibly have nice consequences in the context of pre-Lie deformation theory, as developed in \cite{DSV}.

\subsection*{Acknowledgements}

The idea of this paper emerged after both authors participated at the summer school in Deformation Theory held in Turin in 2017: we are thus thankful to the organizers. Moreover, we are grateful to M. Lehn and M. Manetti, who gave very inspiring mini-courses on formality criteria for associative and Lie algebras in Turin. We want to thank Z. Zhang for the helpful discussions on Kaledin's results. The first author also thanks E. Martinengo for answering further questions about Kaledin's work and its applications. Finally, we thank Dan Petersen, Andreas Hochenegger and the anonymous referee for their useful comments.

The first author has been partially supported by FIRB2012 "Moduli spaces and applications". The second author was sponsored by Nero Budur's research project G0B2115N from the Research Foundation of Flanders.

\section{Recollections on operads and homotopy algebras}\label{sect:preliminaries}

This section serves as a background to our results as well as fixing notation; however, we assume familiarity with chapters 5 to 7 in \cite{LV}.

\subsection{Notations and conventions}

We will work over a field $\bK$ of characteristic $0$.

The category of complexes of $\bK$-modules will be denoted $\dgMod$, and its objects will simply be called dg modules. We will adopt the cohomological point of view, and say that the differential of a dg module $V \in \dgMod$ is of degree $+1$. If $V \in \dgMod$ is a dg module, and $x \in V$ is a homogeneous element, we will denote by $\overline{x}$ the cohomological degree of $x$.  

\subsection{Operads and cooperads}\label{sect:operads}

In this paper we are interested in formality criteria for algebras over dg Koszul operads. These are quadratic operads under some acyclicity conditions (see \cite[Theorem 7.4.2]{LV}). We briefly recall some properties of these operads. 

We start with $\cP(E,R)$ being a \emph{quadratic operad}; i.e. an operad generated by operations of arity $2$ associated to the quadratic data $(E,R)$ (see \cite[Section 7.1]{LV}). An important feature of quadratic (co)operads is that they can be endowed with a natural weight grading (see \cite[Sections 7.1.2 and 7.1.3]{LV}). Weight grading becomes important in the next sections, especially in relation to spectral sequences.

More specifically, a weight grading on a (co)operad $\cP$ ($\cC$) is a decomposition
\[\cP = \bK\id \oplus \cP^{(1)}\oplus\dots \oplus \cP^{(w)}\oplus \dots,\]
where both the differential and the (de)composition map are compatible with the total weight. Furthermore, morphisms between such (co)operads also preserve the weight.

\begin{remark}
An operad $\cP$ such that $\cP(0)=0$ is called \emph{reduced}. Such operads have the following canonical weight grading:
\[\cP^{(k)}:=\cP(k+1).\]
This weight grading can be carried over to $\cP$-algebras. In particular, $\cP(V)$ can be written as a decomposition $\cP(V)^{(w)}:= \cP^{(w)}\circ (V,0, 0, \dots)$, see \cite[Section 5.7.1]{LV} for more details. The same holds for reduced cooperads and their coalgebras.

\end{remark}

As mentioned in the introduction, all the operads that we consider are reduced.

\subsection{Homotopy \texorpdfstring{$\cP$}~-algebras}

Let $\cP$ be a dg Koszul operad as in \cite[Section 7.4.3]{LV}. Moreover, we will suppose that the generating operations of $\cP$ are of cohomological degree $0$, as this simplifies some of our arguments. Now, consider the Koszul dual cooperad $\cC=\cP^\antishriek$ of $\cP$, as described in \cite[Section 7.2.1]{LV}. The cooperad $\cC$ plays a fundamental role in the theory of homotopy algebras.

\begin{defi}
A \emph{homotopy $\cP$-algebra} is an algebra over the Koszul dual operad $\Omega\cC$.
\end{defi}

Alternatively, homotopy $\cP$-algebras will also be called \emph{$\cP_\infty$-algebras}, where $\cP_\infty$ stands for the canonical resolution $\Omega\cC$ of $\cP$. 

Let now $A \in \dgMod$ be a dg module.
Then a $\cP_{\infty}$-structure on $A$ is the same as a codifferential of degree $+1$ on the cofree $\cC$-coalgebra $\cC(A)$ cogenerated by $A[1]$ (see \cite[Section 10.1.8]{LV}). For example, for $\cP=Lie$, a $Lie_\infty$ (or just $L_\infty$) algebra structure on $A$ is the datum of a degree $+1$ coderivation $Q$ on the cofree coassociative coalgebra cogenerated by $A[1]$.

Note that $\cC(A)$ is cofree as a $\cC$-coalgebra, and thus we have the following isomorphism by \cite[Proposition 6.3.8]{LV}:
\[ \Coder(\cC(A)) \simeq \Hom(\cC(A), A),\]
where the right hand side is the internal hom of complexes. Now, given that $\Coder(\cC(A))$ can be seen as a graded Lie algebra, we can endow $\Hom(\cC(A), A)$ with a graded Lie bracket. In the case of $\cP = \Lie$, this is commonly known as Nijenhuis-Richardson bracket.

An important observation here is that $\cC(A)$ is a \emph{graded complex}: by this we mean that it has an \emph{internal} grading, coming from the fact that $A$ has a cohomological grading itself, but also an \emph{external} grading, coming from the fact that it is a cofree coalgebra. To avoid confusion, we will refer to this external grading by calling it \emph{weight} (as it is induced by the canonical weight grading on the cooperad $\cC$), while the internal grading will be called \emph{degree}. Finally, the weight grading on $\cC(A)$ automatically gives an additional grading on $\Hom(\cC(A),A)$.

The induced weight grading on $\Coder(\cC(A))$ is compatible with the Lie bracket of coderivations. More specifically, the Lie bracket can be seen as a morphism of graded complexes
\[ \Coder(\cC(A)) \otimes \Coder(\cC(A)) \to \Coder(\cC(A)). \]

Consider now a codifferential $Q \in \Coder(\cC(A))$, that is to say a coderivation of cohomological degree 1 which squares to zero. With such $Q$, the dg module $A$ becomes a $\cP_\infty$-algebra, which will be denoted by $(A,Q)$. Since $\cC^{(0)}=I$, the weight zero component of $Q$ must be zero itself, and thus $Q$ can be thought of as an infinite sum
\[ Q = q_1 + q_2 + \dots  \]
where $q_i$ has weight $i$.

The relevance of the weight grading on $\Coder(\cC(A))$ can be seen in the following result.

\begin{prop}
Let $(A,Q)$ be a $\cP_\infty$-algebra. Then the following are equivalent:
\begin{enumerate}
\item the algebra $(A,Q)$ is a strict $\cP$-algebra;
\item the codifferential $Q$ is concentrated in weight 1.
\end{enumerate}
\end{prop}

\begin{proof}
This is an immediate consequence of \cite[Proposition 10.1.4]{LV}.
\end{proof}

The interpretation of $\cP_\infty$-structures in terms of codifferentials is particularly nicely suited to define $\cP_\infty$-morphisms.

\begin{defi}
A morphism of $\cP_{\infty}$-algebras $f : (A,Q) \to (B,R)$ is a map of $\cC$-coalgebras $\cC(A) \to \cC(B)$ commuting with the differentials $Q$ and $R$.
\end{defi}

Note that $\cP_\infty$-morphisms also have a weight grading decomposition. More specifically, every morphism of $\cC$-coalgebras $\cC(A) \to \cC(B)$ is completely determined by the composite
\[ \cC(A) \to \cC(B) \to B, \]
as $\cC(B)$ is cofree. Let us denote by $\End_B^A$ the $\bS$-module defined by
\[ \End_B^A(n) := \Hom(A^{\otimes n}, B). \]
Then a linear map $\cC(A) \to B$ is equivalent to an element of $\Hom_{\bS}(\cC, \End_B^A)$, and as before the weight grading on $\cC$ induces a decomposition
\[ f=(f_0,f_1,\dots). \] 
In particular, the component $f_0$ gives a map of complexes $A \to B$.

\begin{defi}
We say that a map $f$ of $\cP_{\infty}$-algebras is a quasi-isomorphism if the induced map $f_0:A \to B$ is a quasi-isomorphism of complexes in the usual sense.
\end{defi}

\subsection{Filtrations}\label{sect:filtration}
Let $A$ be a $\cP_\infty$-algebra. In what follows, we will use the natural filtration on $\Coder(\cC(A))$ induced by the weight grading. It is actually convenient to construct it in the slightly more general context of $\cP_\infty$-morphisms, as follows.

Take a $\cP_{\infty}$-morphism $f: (A,Q) \to (B,R)$. This gives a map $\cC(A) \to \cC(B)$ of $\cC$-coalgebras, and by viewing $\cC(A)$ as a $\cC(B)$-comodule we can construct the complex
\[ \Coder(\cC(A), \cC(B); f). \]
The differential in here is induced by the two differentials $Q$ and $R$ on $\cC(A)$ and $\cC(B)$ respectively.
In the case of $\cP=\Lie$, this is precisely the Chevalley-Eilenberg complex of a $L_{\infty}$-map, as constructed in \cite[Definition 5.2]{Ma}.

Since $\cC(B)$ is cofree, every coderivation $\cC(A) \to \cC(B)$ is completely determined by the composition
\[ \cC(A) \to \cC(B) \to B.\]
Moreover, as mentioned there is a weight grading decomposition
\[ \cC(A) = \bigoplus_{i \geq 0} \cC(A)^{(i)} \]
on the cofree coalgebra $\cC(A)$, which in turn induces a decomposition
\[ \Coder(\cC(A), \cC(B);f) = \prod_{i \geq 0} \Coder(\cC(A), \cC(B), f)^{(i)}. \]
Therefore, we define a filtration
\[ F^p\Coder(\cC(A), \cC(B);f) := \prod_{i \geq p} \Coder(\cC(A), \cC(B), f)^{(i)}. \]

In the special case where $f$ is the identity of a $\cP_\infty$-algebra $(A,Q)$, the corresponding filtration on $\Coder(\cC(A))$ will be denoted by $F^p\Coder(\cC(A))$.

\section{Formality of  \texorpdfstring{$\cP_\infty$}--algebras}

In this section we introduce the important notion of \emph{formal} $\cP_\infty$-algebras. Our first main result is that formality is controlled by a certain cohomology class, which we call \emph{operadic Kaledin class}.

\subsection{Minimal and formal homotopy algebras}
We start by the simplest notion of minimality for algebras over operads.

\begin{defi}
Let $\cP$ be any operad, and let $A$ be a $\cP$-algebra. We say that $A$ is \emph{minimal} if its differential is zero.
\end{defi}

Notice that if $(A,Q)$ is a minimal $\cP_\infty$-algebra, and $Q=q_1+q_2+\dots $ is the weight decomposition of $Q$, then $(A,q_1)$ is a minimal $\cP$-algebra. 

\begin{lem}\label{lem:minimal}
Every $\cP_{\infty}$-algebra $(A,Q)$ is quasi-isomorphic to a minimal $\cP_\infty$-algebra.
\end{lem}
\begin{proof} This is an immediate consequence of \cite[Theorem 10.3.10]{LV}. Alternatively, we know that $A$ is quasi-isomorphic to its cohomology $H(A)$ as dg modules over $\bK$. It follows that there is an induced $\cP_\infty$-structure $Q'$ on $H(A)$ such that $(H(A),Q') \simeq (A,Q)$ as $\cP_\infty$-algebras.
\end{proof}

In veiw of Lemma \ref{lem:minimal}, we deduce following guiding principle: if we are only interested in constructions and properties of $\cP_\infty$-algebras which are invariant under quasi-isomorphisms, then we can safely restrict ourselves to minimal $\cP_\infty$-algebras.

\begin{defi}
A $\cP_{\infty}$-algebra $A$ is said to be \emph{formal} if it is quasi isomorphic to a strict $\cP$-algebra, which is moreover minimal.
\end{defi}

\begin{remark}
Since we are working over a field, Lemma \ref{lem:minimal} tells us that any $\cP_\infty$-algebra $(A,Q)$ is quasi-isomorphic to the $\cP_\infty$-algebra $(H(A),Q')$. On the other hand, the differential on $H(A)$ is of course zero, and thus $H(A)$ is also a strict $\cP$-algebra. However, the quasi-isomorphism between $A$ and $H(A)$ is not compatible with this strict $\cP$-structure.
\end{remark}

We now address the question of finding a cohomological obstruction to formality for $\cP_\infty$-algebras. As a consequence of Lemma \ref{lem:minimal}, we can restrict our attention to minimal $\cP_\infty$-algebras.  

Let thus $(A,Q)$ be a minimal $\cP_\infty$-algebras. In other words, let $Q=q_1+q_2+\dots$ be a Maurer-Cartan element in the graded Lie algebra $\Coder(\cC(A))$. Then in particular we can use $Q$ to define a differential $d_Q := [Q,-]$ on $\Coder(\cC(A))$.
The cohomology of the complex $(\Coder(\cC(A)), d_Q)$ is the \emph{operadic cohomology} of $A$.

\begin{remark}
If for example $\cP=Lie$, we get back the usual Chevalley-Eilenberg cohomology. In other interesting cases such as $\cP= \Ass$, $\cP= \Comm$ or $\cP = \Pois$ we obtain the standard notion of Hochschild cohomology, Harrison cohomology and Poisson cohomology respectively.
\end{remark}

\subsection{Operadic Kaledin class and formality}\label{sect:kaledin}
Again, let $(A,Q)$ be a minimal $\cP_\infty$-algebra, and consider the weight decomposition $Q= q_1 + q_2 + \dots$ of $Q$ inside $\Coder(\cC(A))$. If we define
\[ \widetilde{Q} := q_2 + 2q_3 + 3q_4 + \dots, \]
then a simple check gives us that $[Q, \widetilde{Q}]=0$. Moreover, clearly $\widetilde{Q}$ is a degree 1 element of $F^1\Coder(\cC(A))$, and thus $\widetilde{Q}$ defines a class in $H^1(F^1\Coder(\cC(A)))$, where the cohomology is computed with respect to the differential $d_Q$.

\begin{remark}
The additional grading on the complex $\Coder(\cC(A))$ can be interpreted as a structure of a $k[h,h^{-1}]$-comodule, where $h$ is a formal variable. In this sense, $\widetilde{Q}$ can be thought of as a sort of formal derivative of $Q$ with respect to the formal variable $h$, which controls the weight grading. This is essentially the approach taken in \cite{Ka} and \cite{Lu}.
\end{remark}

\begin{defi}\label{defi:kaledinclass}
The class $K_A \in H^1(F^1(\Coder(\cC(A))))$ defined by $\widetilde{Q}$ is the \emph{operadic Kaledin class} of the $\cP_\infty$-algebra $A$.  
\end{defi}

\begin{remark}\label{rmk:F^1}
Notice that $\widetilde{Q}$ also clearly defines a class in the cohomology of the whole complex $\Coder(\cC(A))$, which includes the weight 0 component. However, this class turns out to be identically zero, as we will show in Section \ref{sect:euler}.
\end{remark}

The $n$-truncation $K_A^{\leq n}$ of $K_A$ is obtained by considering only weight components of weights $\leq n$. More explicitly, the element
\[ \widetilde{Q}^{\leq n} = q_2 + 2q_3 + \dots + (n-1)q_n \]
is a cocycle in the complex $F^1\Coder(\cC(A))/F^{n+1}$, and we set
\[ K_A^{\leq n} = [q_2 + 2q_3 + \dots (n-1)q_n] \in H^1(F^1\Coder(\cC(A))/F^{n+1}).  \]

The importance of the operadic Kaledin class is that it controls formality of $A$.

\begin{prop}\label{prop:n-formality}
Let $(A,Q)$ be a minimal $\cP_\infty$-algebra, and let $K_A$ be its Kaledin class. The following are equivalent:
\begin{enumerate}
\item there exists an isomorphism of minimal $\cP_\infty$-algebras $(A,Q) \to (A,R)$, where $r_1=q_1$ and $r_2=r_3 = \dots = r_n=0$;
\item the truncated class $K^{\leq n}_A$ is zero. 
\end{enumerate}
\end{prop}

Before proving the Proposition, we need a preliminary lemma on invariance of the Kaledin class under isomorphisms of minimal $\cP_{\infty}$-algebras.

\begin{lem}\label{lem:kalisom}
	Let $f: (A,Q) \to (A,R)$ be an isomorphism between minimal $\cP_{\infty}$-algebras. Then $f$ induces a natural map $H^1(F^1\Coder(\cC(A)), [Q,-]) \to H^1(F^1\Coder(\cC(A)), [R,-]) $, which we still denote by $f$, abusing notation. 
	Then $f(K_{(A,Q)}) = K_{(A,R)}$.
\end{lem} 

\begin{proof}
	By definition, $f$ is determined by the sequence $(f_0,f_1,\dots)$, where every $f_i: \cC(A)^{(i)}\to A$ is a linear map. Since $f$ is an isomorphism, $f_0$ is a linear automorphism of $A$. Let $g$ be the coalgebra automorphism $\cC(A) \to \cC(A)$ whose projection on cogenerators is simply given by $f_0$ (i.e. $g$ is pure of weight $0$). On the other hand, let $h$ be the coalgebra automorphism whose projection on cogenerators is given by the sequence $(\id, f_1f_0^{-1}, f_2f_0^{-1},\dots)$. It is immediate to check that $h\circ g = f$. In other terms, setting $S = g \circ Q \circ g^{-1}$, we have two isomorphisms of minimal $\cP_{\infty}$-algebras
	\[	\xymatrix{
	(A,Q) \ar[r]^{g} & (A,S) \ar[r]^{h} & (A,R),  
	}	\] 
	whose composition is $f$.
	
	As $g$ is concentrated in weight zero, it is straightforward to see that it preserves the Kaledin class. Therefore to conclude it is now sufficient to prove that also $h$ preserves Kaledin classes: this is essentially the content of \cite[Lemmas 4.1 and Proposition 7.3]{Lu}. 
\end{proof}

\begin{proof}[Proof of Proposition \ref{prop:n-formality}]
The fact that $(1)$ implies $(2)$ is clear, since the $n$-truncation of the Kaledin class of $(A,R)$ is zero, and the Kaledin class is invariant under isomorphisms by Lemma \ref{lem:kalisom}.
In order to prove $(2) \Rightarrow (1)$, we can work by induction. The case $n=1$ is clear. Suppose now that the proposition holds for $n-1$ and that $K_A^{\leq n}=0$. In particular, also $K_A^{\leq(n-1)}=0$, and by induction hypothesis, without loss of generality we can safely suppose that $q_2=\dots= q_{n-1}=0$. In this case we have
\[ Q = q_1 + q_n + q_{n+1} + \dots \]
and therefore
\[ \widetilde{Q} = (n-1)q_n + nq_{n+1} + \dots. \]
Since $\widetilde{Q}$ must be zero in cohomology in weights $\leq n$, there exists a $T = t_1 + t_2 + \dots \in F^1\Coder(\cC(A))$ such that
\[ [Q, T] = \widetilde{Q} \]
in weights $\leq n$.
By looking at what happens in weight $n$, we deduce that we have
\[ [q_1 , t_{n-1}] = (n-1)q_n . \]
Putting $\tau= \frac{t_{n-1}}{n-1}$, we get $[q_1, \tau] = q_n$. Consider $R = e^{-\tau} Q e^{\tau} \in F^1\Coder(\cC(A)).$ Then we get that $[R,R]=0$, and thus $(A,R)$ is a $\cP_\infty$-algebra.

Moreover, by definition $e^{\tau}$ gives a $\cP_\infty$-isomorphism $(A,R) \to (A,Q)$.
Finally, it is straightforward to check that indeed $r_i=q_i$ for $i < n$, and that $r_n = q_n - [q_1, \tau] = 0$, which concludes the proof.
\end{proof}

\begin{remark}
	In the proof of Proposition \ref{prop:n-formality} we crucially use the fact that the Kaledin class lives in $H^1(F^1(\Coder(\cC(A))))$ rather that in $H^1(\Coder(\cC(A)))$. In fact, in order to get well-defined exponentials, it is essential that the element $T$ appearing in the proof has no component in weight 0.
\end{remark}

From the above Proposition \ref{prop:n-formality} we can deduce our first main result.

\begin{thm}\label{thm:formality1}
A minimal $\cP_\infty$-algebra $(A,Q)$ is formal if and only if its Kaledin class is $K_A$ is zero.
\end{thm}
\begin{proof}
One direction is straightforward: if $A$ is formal, then there exists an isomorphims
\[ (A, q_1) \lra (A, Q) \]
and by Proposition \ref{prop:n-formality} we deduce that $K_A$ is zero, since all its truncations are zero.
For the other direction, suppose that $K_A=0$. Let $i$ be the smallest integer $i\geq 2$ such that $q_i \neq 0$. Then by Proposition \ref{prop:n-formality} we know that there is an isomorphism
\[ f_i = e^{\tau_i} : (A, R_i) \to (A,Q) \]
where $r_2=\dots = r_i=0$, and $r_1=q_1$. Repeating the procedure, we can get elements $\tau_j$ for every $j\geq i$. Now consider
\[ f:= \dots f_{i+1}f_i =\dots e^{\tau_{i+1}} e^{\tau_i}. \]
The infinite composition here makes sense, as every $\tau_j$ is a coderivation of weight $j-1$, and therefore one sees that $e^{\tau_j}$ is the identity in weights smaller that $j-1$. In particular, for every weight component, the infinite composition reduces to a finite product. 
It follows that $f$ gives the desired isomorphism
\[ (A, q_1) \lra (A, Q). \]
\end{proof}

\section{Operadic Euler classes}

The goal of this section is to present an alternative approach to formality. In particular we will prove a reformulation of the formality criteria for a minimal $\cP_{\infty}$-algebra $A$ of Theorem \ref{thm:formality1}, expressed in terms of its operadic cohomology spectral sequence.

\subsection{Euler derivations}\label{sect:euler}

In this section we expand Remark \ref{rmk:F^1}. More specifically, we introduce the important notion of \emph{operadic Euler class}, generalizing the Euler derivation of \cite[Section 5]{Ma}.

\begin{defi}\label{defi:euler}
The \emph{operadic Euler derivation} $e_f$ of a $\cP_\infty$-morphism $f: (A,Q) \to (B,R)$ is the map of graded $\bK$-modules $A \to B$
defined by
\[ e_f(a) = (\overline{a}+1)f_0(a), \]
where $a \in A$ is homogeneous, and $f_0$ is the induced map of complexes
\[ f_0 \colon A \to B. \]
\end{defi}

Notice that in general $e_f$ is not compatible with the differentials of $A$ and $B$. However, if both $A$ and $B$ are minimal $\cP_\infty$-algebras, then we can consider $e_f$ as a map of complexes.

\begin{defi}
The Euler derivation $e_A$ of a $\cP_\infty$-algebra $(A,Q)$ is the Euler derivation of the identity of $A$.
\end{defi}

Suppose now that $A$ is a minimal $\cP_\infty$-algebra.
The map $e_A$ is a $\bK$-linear map $A\to A$, and as such can be regarded as a weight 0 element (which we still denote by $e_A$) in $\Coder(\cC(A))$. For our purposes, the main property of $e_A$ is given by the following important lemma. 

\begin{lem}\label{lem:eulerbracket}
Let $\beta \in \Coder(\cC(A))^{(p)}$ be a weight $p$ coderivation. Then we have
\[ [\beta, e_A] =( p- \overline{\beta} ) \beta, \]
where as always $\overline{\beta}$ is the cohomological degree of $\beta$.
\end{lem}
\begin{proof}
This is a straightforward verification, using the explicit definition of the bracket on the complex $\Hom(\cC(A), A) \simeq \Coder(\cC(A))$, as given for example in \cite[Proposition 6.4.3]{LV}.
\end{proof}

As an immediate consequence of Lemma \ref{lem:eulerbracket}, we find that in $\Coder(\cC(A))$ we have
\[ [q, e_A] = [ q_1 + q_2 + \dots , e_A] = \sum_{i \geq 1}\left[q_i, e_A\right] = \sum_{i\geq 1}(i-1)q_i = \widetilde{Q}. \]
In particular, the cohomology class of $\widetilde{Q}$ in $\Coder(\cC(A))$ is always zero. However, $e_A$ lives in weight $0$, and hence doesn't tell us much about $K_A$, which is a cohomology class of $F^1\Coder(\cC(A))$.
Nevertheless, the Euler class $e_A$ can be used to formulate an alternative formality criterion based on spectral sequences.

\subsection{Operadic cohomology spectral sequences}\label{sect:spectralseq}

Let $f: A \to B$ be a $\cP_\infty$-map. We saw in Section \ref{sect:filtration} that the complex $\Coder(\cC(A), \cC(B); f)$ carries a natural filtration, induced by the weight grading. In this section we study the associated spectral sequence. 

\begin{defi}\label{defi:spectralseq}
The operadic cohomology spectral sequence of a $\cP_\infty$-map $f: A\to B$ is the spectral sequence $(E(A,B;f)^{p,q}_r, d_r)$ associated to the filtered complex $\Coder(\cC(A), \cC(B); f)$.
\end{defi}

Our first goal is to show that these spectral sequences are functorial, in the sense of the following result.

\begin{lem}\label{lem:ssfunct}
Let $f: (A,Q) \to (B,R)$ and $g : (B,R) \to (C,S)$ be  $\cP_\infty$-morphisms. Then the composition maps produces morphisms
\[ g_* \colon \Coder(\cC(A), \cC(B); f) \to \Coder(\cC(A), \cC(C); gf) \]
\[ f^* \colon \Coder(\cC(B), \cC(C); g) \to \Coder(\cC(A), \cC(C); gf) \]
of filtered complexes.
\end{lem}

\begin{proof}
This is a straightforward check. In fact, it suffice to verify that both $g_*$ and $f^*$ are compatible with the filtrations, and that they commute with differentials.
\end{proof}

The next lemma asserts that the spectral sequence $E(A,B;f)^{p,q}_r$ has a nice homotopy invariance property with respect to weak equivalences.

\begin{prop}\label{prop:sshominv}
Let again $f: (A,Q) \to (B,R)$ and $g : (B,R) \to (C,S)$ be $\cP_\infty$-morphisms, which by Lemma \ref{lem:ssfunct} induce morphisms
\[ \xymatrix{
E(A,B;f)^{p,q}_r \ar[r]^{g_*} & E(A,C;gf)^{p,q}_r & E(B,C;g)^{p,q}_r \ar[l]_{\ \ f^*}
} \]
of spectral sequences.
If $g$ is a weak equivalence, then $g_*$ is an isomorphism for every $r \geq 1$. Similarly, if $f$ is a weak equivalence, then $f^*$ is an isomorphism for every $r \geq 1$.
\end{prop}

\begin{proof}
It is enough to prove the claim for $r=1$. By definition, we have
\[ E(A,B;f)_0^{p,q} = \Hom^{p+q}(\cC(A)^{(p)},B), \]
where the right hand side denotes the degree $p+q$ maps between $\cC(A)^{(p)}$ and $B$. Moreover, the differential $d_0$ on the $0$-th page of the spectral sequence is precisely the one induced by the differentials on $\cC(A)^{(p)}$ and $B$. By the K\"unneth formula, we deduce that the first page has the form
\[ E(A,B;f)_1^{p,q} = \Hom^{p+q}(\cC(H(A))^{(p)}, H(B)), \]
and thus similarly we have
\[ E(A,C;gf)_1^{p,q} = \Hom^{p+q}(\cC(H(A))^{(p)}, H(C)), \]
It is now clear that if $g$ is a weak equivalence, then $g_*$ provides an isomorphisms between the two spectral sequences. The case of $f$ being a weak equivalence is also completely analogous.
\end{proof}

\subsection{Operadic Euler classes}
Let again $f : A \to B$ be a $\cP_\infty$-morphism. Notice that the Euler derivation of Section \ref{sect:euler} can be regarded as living in the first page of the spectral sequence $E(A,B;f)$. More specifically, we have
\[ e_f \in E(A,B;f)_1^{0,0} = \Hom^0(H(A),H(B)). \]

\begin{prop}\label{prop:ssminmod}
	Let $A$ be a $\cP_\infty$-algebra, and let $M$ be its minimal model. Then there is a morphism of spectral sequences
	\[  E(A,A)_r^{p,q} \to E(M,M)_r^{p,q}	\]
	which moreover is an isomorphism for $r\geq 1$ and preserves the Euler derivation.
\end{prop}

\begin{proof}
	As $M$ is the minimal model of $A$, one can find two weak equivalences of $\cP_{\infty}$-algebras
	\[	f: A \to M,\ \ \ g: M \to A,	\]
	such that $fg$ is the identity of $M$. Applying Proposition \ref{prop:sshominv}, we get that both arrows in the diagram of spectral sequences
	\[ \xymatrix{
		E(A,A)^{p,q}_r \ar[r]^{g^*} & E(M,A;g)^{p,q}_r  \ar[r]^{f^*\ \ \ \ \ \ \ \ \ \ \ } & E(M,M;fg)^{p,q}_r = E(M,M)^{p,q}_r
	} \]
	are isomorphisms as soon as $r \geq 1$, and their composition is easily seen to preserve Euler derivations.
\end{proof}

The following Lemma is a generalization of a similar result of Manetti, namely \cite[Lemma 5.8]{Ma}.

\begin{lem}
The Euler derivation of any $\cP_\infty$-map $f : (A,Q) \to (B,R)$ satisfies $d_1(e_f)=0$.
\end{lem}
\begin{proof}
Using Proposition \ref{prop:sshominv}, we can safely assume that both $A$ and $B$ are in fact minimal $\cP_\infty$-algebras. 

If $\phi:A \to B$ is any $k$-linear map, then we can look at $\phi$ as an element in
\[ E(A,B;f)^{0,0}_1 = \Hom^0(A,B). \]
In other terms, $\phi$ induces a well-defined coderivation $\hat{\phi}: \cC(A) \to \cC(B)$. By definition, the value of $d_1(\phi)$ is the weight 1 component of the coderivation given by
\[ R \circ \hat{\phi} - \hat{\phi} \circ Q, \]
where $R$ and $Q$ are the (co)differentials on $\cC(B)$ and $\cC(A)$ respectively.

Recall that the weight 1 component of a coderivation $\cC(A) \to \cC(B)$ is given by a linear map $\cC(A)^{(1)} \to B$, 
or equivalently by a map of $\bS$-modules
\[ \cC^{(1)} \to \End_B^A. \]
Since $\cP=\cP(E,R)$, we have by definition that $\cC=\cP^{\antishriek}= \cC(sE, s^2R)$, and thus $\cC^{(1)}$ can be identified with $E$.

Let $c \in E$ be a generating (co)operation of $\cC$ (which is of co-arity 2 and cohomological degree 0), and take $a \in A$. 
Then in the special case of $\phi$ being the Euler derivation $e_f$, the weight 1 component of $  \widehat{e_f} \circ Q$ is simply given by 
\[ \begin{array}{ccc}
E \otimes A^{\otimes 2} & \lra & B \\
c \otimes a_1 \otimes a_2 & \longmapsto & e_f(q_1(c\otimes a_1 \otimes a_2)),
\end{array} \]
where we are identifying $q_1$ with the corresponding map $E \otimes A^{\otimes 2} \to A$. Moreover, we have
  \[ e_f(q_1(c \otimes a_1 \otimes a_2)) = (\overline{a_1} + \overline{a_2}+2) f_0(q_1(c \otimes a_1 \otimes a_2)),
  \]
where we used that $q_1$ has cohomological degree 1.

On the other hand, $\widehat{e}_f$ is the coderivation extending $e_f$, and thus we have that the weight 1 component of $R \circ \widehat{e_f}$ can also be computed explicitly as
\[ \begin{array}{ccc}
E \otimes A^{\otimes 2} & \lra & B \\
c \otimes a_1 \otimes a_2 & \longmapsto & r_1(c\otimes e_f(a_1) \otimes f_0(a_2)) + r_1 (c \otimes f_0(a_1) \otimes e_f(a_2)).
\end{array} \] 
Now a straightforward check shows that in this case $R \circ \hat{\phi}$ and $\hat{\phi} \circ Q$ have the same weight 1 component, and this concludes the proof.

\end{proof}

\begin{remark}
In the special case of $f$ being the identity of a minimal $\cP_\infty$-algebra $(A,Q)$, the fact that $d_1(e_A)=0$ was already implicit in Lemma \ref{lem:eulerbracket}.
\end{remark}

It follows in particular that $e_f$ defines an element in $E(A,B;f)^{0,0}_2$.

\begin{defi}
The \emph{Euler class} of a $\cP_\infty$-morphism $f: (A,Q) \to (B,R)$ is the class of the Euler derivation in $E(A,B;f)^{0,0}_2$. The Euler class $e_A$ of a single $\cP_\infty$-algebra $(A,Q)$ is simply the Euler class of the identity.
\end{defi}

Notice that the Euler class of a $\cP_\infty$-algebra is invariant under quasi-isomorphisms: thanks to Proposition \ref{prop:ssminmod} we can restrict our attention to Euler classes of minimal models, and if $A$ and $B$ are quasi-isomorphic $\cP_\infty$-algebras, their minimal models are isomorphic.

\section{An alternative formality criterion for  \texorpdfstring{$\cP_\infty$}--algebras}

Recall that Theorem \ref{thm:formality1} relates formality of a minimal $\cP_\infty$-algebra $(A,Q)$ with the vanishing of a certain cohomology class $K_A \in H^1(F^1\Coder(\cC(A)), [Q,-])$. Also, recall from Section \ref{sect:filtration} that the complex $\Coder(\cC(A))$ carries a natural filtration, for which we considered the associated cohomology spectral sequence $E(A)^{p,q}_r$. As usual, the differential in the $r$-th page of the spectral sequence $E(A)^{p,q}_r$ will be denoted by $d_r$.

\begin{thm}\label{thm:kaledin-euler}
Let $(A,Q)$ be a minimal $\cP_\infty$-algebra. Then the following are equivalent:
\begin{enumerate}
\item The truncated Kaledin class $K_A^{\leq n}$ is zero;
\item $d_r(e_A)=0$ for $r=2,\dots,n$.
\end{enumerate}
\end{thm}

\begin{remark}
	Item $(2)$ in Theorem \ref{thm:kaledin-euler} contains a slight abuse of notation. In fact, its content is first of all that $d_2(e_A)=0$, and thus $e_A$ defines an element in the third page $E_3$ of the operadic cohomology spectral sequence. If we keep denoting it with $e_A$, we are asking that $d_3(e_A)$ also vanishes, so that $e_A$ in turn defines an element in the fourth page $E_4$, and so on and so forth.
\end{remark}

\begin{proof}
Suppose that $K_A^{\leq n}=0$. By Proposition \ref{prop:n-formality} we can assume that $q_2=q_3= \dots = q_n=0$.

Let us write $Q=q_1 + q'$, where $q' \in F^{n+1}\Coder(\cC(A))$. Suppose $1 < r \leq n$, and take $x \in \Coder(\cC(A))^{(p)}$ such that $[Q,x] \in F^{p+r}\Coder(\cC(A))$. Since $r > 1$, we get $[q_1, x] = 0$ by weight reasons, and hence
\[ [Q,x] = [q',x] \in F^{p+n+1}\Coder(\cC(A)) \subset F^{p+r+1}\Coder(\cC(A)). \]
This yields $d_r=0$, and in particular proves that $(1) \Rightarrow (2)$.

Let us prove that $(2) \Rightarrow (1)$. Recall that the differential of the $r$-th page of the spectral sequence
\[ d_r : \frac{Z^{0}_r}{Z_{r-1}^{1} + dZ_{r-1}^1} \lra \frac{Z^{r}_r}{Z_{r-1}^{r+1} + d Z^1_{r-1}} \]
is induced by the differential $d_Q = [Q,-]$ on $\Coder(\cC(A))$. The fact that $d_2(e_A)=0$ means that
\[ d_Q(e_A) = [Q,e_A] \equiv 0 \ (\textrm{mod}\ Z_{1}^{2} + dZ_{1}^1), \]
or equivalently that the weight 2 component of $[Q,e_A]$ belongs to the image of $[q_1,-]$. In other words, there exists an element $t^1_1 \in \Coder(\cC(A))^{(1)}$ of cohomological degree $0$ such that
\[ [q_1, t^1_1] = [q_2,e_A] = q_2. \] 
Similarly, $d_3(e_A)=0$ implies that we can find $t^2 = t^2_1 + t^2_2 \in F^1\Coder(\cC(A))/F^3$ such that
\[ [q_1, t^2_1] = 0, \ \ \ [q_1, t^2_2] + [q_2, t^2_1] = [q_3, e_A] = 2q_3. \]
The same argument shows that for $1 \leq i < n$, we can find 
\[ t^i = t^i_1 + t^i_2 + \dots + t^i_i \in F^1\Coder(\cC(A))/F^{i+1} \]
such that
\[ \sum_{j=1}^{p-1}\left[q_j, t^i_{p-j}\right] = \begin{cases}
0 \ & \textrm{if } 1<p\leq i \\
i q_{i+1} \ & \textrm{if } p= i+1.
\end{cases}
\]
In particular, setting
\[ r_i := \sum_{j=1}^i t^j_i, \ \ \ R:= r_1 + r_2 + \dots r_{n-1}, \]
we get that $[Q,R] \equiv \widetilde{Q} \ (\textrm{mod}\ F^{n+1})$, showing that $K_A^{\leq n}=0$ and thus finishing the proof.
\end{proof}

We summarize our results in the following corollary of Theorem \ref{thm:kaledin-euler}.

\begin{cor}\label{cor:formality2}
Let $(A,Q)$ be a $\cP_\infty$-algebra. Then the following are equivalent:
\begin{enumerate}
\item $A$ is formal;
\item the operadic Kaledin class $K_A$ vanishes;
\item the spectral sequence $E(A)_r^{p,q}$ degenerates at $E_2$;
\item $d_r(e_A) = 0$ for every $r$.
\end{enumerate}
\end{cor}
\begin{proof}
	The fact that $(1)$ is equivalent to $(2)$ is the content of Theorem \ref{thm:formality1}. Moreover, an immediate consequence of Theorem \ref{thm:kaledin-euler} is that condition $(2)$ is equivalent to condition $(4)$. Also, the fact that $(3)$ implies $(4)$ is obvious. It is thus enough to show that $(1)$ implies $(3)$.
	
	We can suppose without loss of generality that $(A, Q)$ is minimal. If $(A,Q)$ is formal, then we can suppose that the only non-zero component of $Q$ is $q_1$. In order to prove $(3)$, we use the following general results on spectral sequences, whose straightforward proof is left to the reader.
	\begin{lem}\label{lem:ssdeg}
		Let $(C,d_C)$ be a cochain complex with an additional weight grading, and suppose that $d_C$ has weight $k$. Then the spectral sequence associated to the natural filtration degenerates at the $(k+1)$-th page. 
	\end{lem}

	Using Lemma \ref{lem:ssdeg}, we immediately see that if $A$ is formal the spectral sequence $E(A)^{p,q}_r$ degenerates at $E_2$, which concludes the proof.
	\end{proof}

\section{Formality transfer}

The goal of this section is to provide an operadic generalization to the formality transfer theorem for dg Lie algebra, which appears as Theorem 6.8 in \cite{Ma}. Our methods here are a straightforward transposition of Manetti's ideas to our operadic context.

We will need the following preliminary lemma, which is a generalization of \cite[Lemma 5.3]{Ma}.

\begin{lem}\label{lem:transfer}
	Let $f: (A,Q) \to (B,R)$ be a $\cP_\infty$-morphism between minimal $\cP_\infty$-algebras, and let $E(A,B,f)^{p,q}_r$ be the associated spectral sequence of Definition \ref{defi:spectralseq}. Denote by
	\[ Q= q_1 + q_2 + \dots \ \ \ R= r_1 + r_2 + \dots \]
	the weight decompositions of $Q$ and $R$. Suppose that for some integer $k \geq 2$ we have
	\[  q_2 = \dots = q_k = r_2 = \dots r_k = 0.  \]
	Then the differential $d_r$ of the $r$-th page of the spectral sequence $E(A,B,f)$ is zero if $1 < r \leq k$.
\end{lem}

\begin{proof}
	This is completely similar to the proof of \cite[Lemma 5.3]{Ma}, and here we only include a proof for completeness. Following Manetti, we work by induction, and it is therefore enough to prove that $d_{k}=0$.
	
	Consider an element $x \in E(A,B,f)^p_{k}$. By definition, $x$ is represented by a weight $p$ element $\alpha \in \Coder(A,B,f)^{(p)}$ such that
	\[	R \circ \alpha - (-1)^{\overline{\alpha}} \alpha \circ Q \in F^{p+k}\Coder(A,B,f),	\]
	where as usual $\overline{\alpha}$ denotes the cohomological degree of $\alpha$. As $k \geq 2$, we have that
	\[  r_1 \circ \alpha - (-1)^{\overline{\alpha}} \alpha \circ q_1 =0  \]
	and hence $R \circ \alpha - (-1)^{\overline{\alpha}} \alpha \circ Q$ has no component in weight $p+k$. It follows that $d_{k}(x)=0$, which concludes the proof.
\end{proof}

Notice in particular that under the hypothesis of the above Lemma, we have that the first pages of the spectral sequence $E(A,B,f)$ are quite simple: more precisely, we get
\[ 		E(A,B,f)_2 = E(A,B,f)_3 = \dots = E(A,B,f)_{k+1}.	   \]

We are now ready to state the main result of this section.

\begin{thm}\label{thm:transfer}
	Let $f: (A,Q) \to (B,R)$ be a $\cP_\infty$-morphism between two $\cP_\infty$-algebras. Suppose that:
	\begin{enumerate}
		\item $B$ is formal;
		\item the induced map $f_* : E(A)_2^{p,1-p} \to E(A,B,f)_2^{p,-p}$ is injective for every $p \geq 2$.
	\end{enumerate}
	Then $A$ is also formal.
\end{thm}

\begin{proof}
	Here we once again closely follow the proof of Theorem 6.8 of \cite{Ma}. Without loss of generality, let us suppose that $A$ is minimal, and that moreover $R=r_1$ is concentrated in weight $1$. As usual, we denote by
	\[   Q= q_1 + q_2 + \dots   \]
	the weight decomposition of $Q$. Thanks to Lemma \ref{lem:ssfunct}, the $\cP_\infty$-morphism $f$ induces two morphisms of spectral sequences
	\[ \xymatrix{
		E(A) \ar[r]^{f_*\ \ \ } & E(A,B;f) & E(B). \ar[l]_{\ \ \ \ \  f^*}
	} \]
	Let $k\geq 2$ be the smallest integer such that $q_k \neq 0$. Using Lemma \ref{lem:transfer}, we know that 
	\[ E(A)^{p,1-p}_2=E(A)^{p,1-p}_k, \ \ \ \ E(A,B,f)^{p,1-p}_2 = E(A,B,f)^{p,1-p}_k,  \]
	and as a consequence we get that also $f_* : E(A)^{p,1-p}_k \to E(A,B,f)^{p,1-p}_k$ is injective for every $p \geq 2$. By looking at the Euler classes $e_A$, $e_B$, $e_f$, we get that $f_*(e_A)=e_f$, and similarly $f^*(e_B)=e_f$. Since $B$ is formal, we can apply the formality criteria $(4)$ of Corollary \ref{cor:formality2}, and therefore get that $d_r(e_B)=0$ for every $r$. It follows that 
	\[ f^*(d_r(e_B)) = d_r(f^*(e_B))= d_r(e_f) = d_r(f_*(e_A)) = f_*(d_r(e_A)) = 0. \]
	Notice that by definition $d_r(e_A) \in E(A)_r^{p,1-p}$. But by the injectivity assumption, we immediately get that $d_r(e_A) = 0$ as soon as $2 \leq r \leq k$. The claim now follows from a straightforward application of Theorem \ref{thm:kaledin-euler} and Theorem \ref{thm:formality1}.
\end{proof}

\section{Examples }

In this section, we look more closely at the special cases where the Koszul operad $\cP$ is the operad of Lie algebras or the operad of associative algebras. In particular, we show how our results are in fact generalizations of the main theorems of \cite{Ka} and \cite{Ma}.

\subsection{Associative algebras}

Let $\cP=\Ass$ be the operad encoding associative algebras. In this case $\cP_\infty$-algebras are called $A_\infty$-algebras. In the paper \cite{Ka}, Kaledin addresses the question of finding obstructions to formality of a given associative algebra $A$. Let us briefly recall, mainly following the exposition presented in \cite{Lu}, the content of Kaledin's paper.

\begin{remark}
	As already mentioned in the introduction, it is important to remark that the arguments of Kaledin work for algebras over a possibly non-trivial base $\bK$-algebra $R$. Geometrically, this corresponds to investigating formality for \emph{families} of $A_\infty$-algebras. We did not attempt to reach that level of generality in the present text, but we certainly feel that our methods can be smoothly adapted to the case of families of $\cP$-algebras, where $\cP$ is any sufficiently nice Koszul operad. We plan to come back to this question in a future work. 
\end{remark}

Let $(A,Q)$ be a minimal $A_\infty$-algebra, and let $h$ be a formal variable. With $Q=q_1+q_2+ \dots$ we denote as usual the weight decomposition of the coderivation $Q$. Kaledin observes that one can construct another minimal $A_\infty$-algebra $(A[h],Q')$, where the weight components of $Q'$ are given by
\[ Q' = q_1 + hq_2 + h^2q_3 \dots \]
It can be checked that $\widetilde{A}:=(A[h], Q')$ is indeed a $\bK[h]$-linear $A_\infty$-algebra.
The algebra $\widetilde{A}$ is an algebraic incarnation of the deformation of $A$ to the normal cone, and it is easy to verify that $A$ is formal if and only if $\widetilde{A}$ is. Notice also that the quotient $\widetilde{A}/h^{n+1}$ is automatically a minimal $\bK[h]/h^{n+1}$-linear $A_\infty$-algebra.

To $\widetilde{A}/h^{n+1}$, Kaledin associates a cohomology class as follows. The $A_\infty$-structure in encoded by a coderivation
\[  Q'_{\widetilde{A}/h^{n+1} } = q_1 + hq_2 + \dots h^nq_{n+1}  \]
which squares to zero. In particular, if we denote by $\delta_h$ the formal derivative with respect to $h$, then  $\delta_h(Q'_{\widetilde{A}/h^{n+1} })$ defines a cocycle in the $\bK[h]/h^n$-linear Hochschild complex of $\widetilde{A}/h^n$. The associated degree 1 element $[\delta_h(Q'_{\widetilde{A}/h^{n+1} })]$ in the Hochschild cohomology is the Kaledin class of $\widetilde{A}/h^{n+1}$.

One of the main results in \cite{Ka} and in \cite{Lu} states that the $A_\infty$-algebra $A$ is formal if and only if all the cohomology classes defined by $\delta_h(Q'_{\widetilde{A}/h^{n+1} })$ are zero for all $n \geq 1$. The following Proposition shows that the more general Definition \ref{defi:kaledinclass} specializes to the notion of Kaledin-Lunts in the case $\cP= \Ass$.

\begin{prop}
	Let $A$ be a minimal $A_\infty$-algebra. Then the class $[\delta_h(Q'_{\widetilde{A}/h^{n+1} })] $ is zero if and only if $K_A^{\leq n} = 0$. 
\end{prop}
\begin{proof}
	The Hochschild cohomology class defined by $\delta_h(Q'_{\widetilde{A}/h^{n+1} })$ is zero if and only if there exists a degree zero element $t$ in the Hochschild complex of $\widetilde{A}/h^n$ such that $[Q'_{\widetilde{A}/h^n}, t]= \delta_h(Q'_{\widetilde{A}/h^{n+1} })$. Let us write
	\[ t = t_1 + ht_2 + \dots h^{n-1}t_{n}\]
	for the decomposition of $t$ in terms of the various powers of $h$. Notice that all the $t_i$'s are degree zero elements in the Hochschild complex of $A$. Let now $s$ be an integer such that $0 \leq s < n$; by looking at the coefficients of $h^s$, we find that the condition $[Q'_{\widetilde{A}/h^n}, t]= \delta_h(Q'_{\widetilde{A}/h^{n+1} })$ is equivalent to the equations
	\[ \sum_{j=0}^s [q_{j+1},t_{s-j}] = (s+1)q_{s+2}  \]
	for every $s$.
	Moreover, by weight considerations we can assume that the weight of $t_j$ is $j$.
	
	But it is now immediate to verify that the conditions satisfied by the elements $t_1,\dots, t_{n}$ are equivalent to the more compact equation
	\[ [Q^{\leq n}, t'] = \widetilde{Q}^{\leq n}, \]
	where $t'$ is the element of $\Coder(\cC(A))$ whose weight decomposition is given by
	\[t' = t_1 + t_1 + \dots + t_{n}, \]
	and $Q^{\leq n}$ and $\widetilde{Q}^{\leq n}$ are as in Section \ref{sect:kaledin}.

\end{proof}

In this sense, Theorem \ref{thm:formality1} specializes to the formality criterion described by Kaledin and Lunts if the base ring is trivial.
We remark however that the formality criteria of items $(3)$ and $(4)$ of Corollary \ref{cor:formality2} were not discussed in the paper \cite{Ka} and \cite{Lu}.

Moreover, the formality transfer criterion of Theorem \ref{thm:transfer} appears to be new in the context of associative algebras.

\subsection{Lie algebras}

Let now $\cP = \Lie$ be the operad of dg Lie algebras. In this case, algebras over $\cP_\infty$ are referred to as $L_\infty$-algebras.
In his paper \cite{Ma}, Manetti studies formality criteria for dg Lie algebras (and $L_\infty$-algebras). As we have mentioned, his results have a somewhat different flavour from those of Kaledin. 

In fact, Manetti starts with studying the Lie-incarnation of the operadic cohomology spectral sequence of Section \ref{sect:spectralseq}, called the Chevalley-Eilenberg spectral sequence. He defines an appropriate Euler class in the Chevalley-Eilenberg spectral sequence, and goes to show that it controls formality of $L_\infty$-algebras. It is straightforward to check that Definition \ref{defi:spectralseq} of operadic cohomology spectral sequences and Definition \ref{defi:euler} of operadic Euler derivations give back the analogous notions described by Manetti.

More specifically, one of the main results of \cite{Ma} is the equivalence of items $(1)$-$(3)$-$(4)$ of Corollary \ref{cor:formality2}, in the particular case where $\cP =\Lie$.
Similarly to what happened in the case of associative algebras, we find that Theorem 6.3 of \cite{Ma} is a consequence of the more general Corollary \ref{cor:formality2}. Moreover, the formality transfer theorem for dg Lie algebras (which is Theorem 6.8 in \cite{Ma}) is a special case of its operadic version contained in Theorem \ref{thm:transfer}.

\end{document}